\documentclass[reqno,twoside,12pt]{amsart}
\hoffset - 1.5cm

\usepackage{amsmath}
\usepackage{amssymb}
\usepackage{graphicx}
\usepackage{amstext}
\usepackage{amsopn}
\usepackage{amsfonts}
\usepackage{xcolor}
\usepackage{setspace}
\usepackage{enumerate}
\usepackage[polish,english]{babel}
\usepackage[utf8]{inputenc}
\usepackage{polski}
\usepackage[all]{xypic}
\usepackage[textwidth=3cm]{todonotes}

\newtheorem{Remark}{Remark}[section]
\newtheorem{Corollary}[Remark]{Corollary}

\newtheorem{Fact}[Remark]{Fact}
\newtheorem{Lemma}[Remark]{Lemma}

\newtheorem{Theorem}[Remark]{Theorem}

\newcommand{\bE}{\mathbb{E}}

\newcommand{\bH}{\mathbb{H}}
\newcommand{\bI}{\mathbb{I}}

\newcommand{\bN}{\mathbb{N}}

\newcommand{\bR}{\mathbb{R}}

\newcommand{\bW}{\mathbb{W}}
\newcommand{\bV}{\mathbb{V}}
\newcommand{\bX}{\mathbb{X}}

\newcommand{\bZ}{\mathbb{Z}}

\newcommand{\cC}{\mathcal{C}}

\renewcommand{\cH}{\mathcal{H}}

\newcommand{\cN}{\mathcal{N}}

\newcommand{\cT}{\mathcal{T}}

\numberwithin{equation}{section} \errorcontextlines=0
\newcommand{\degso}{\nabla_{\sone}\textrm{-}\mathrm{deg}}
\newcommand{\degg}{\nabla_{G}\textrm{-}\mathrm{deg}}
\newcommand{\degh}{\nabla_{H}\textrm{-}\mathrm{deg}}

\newcommand{\diag}{\mathrm{ diag \;}}

\newcommand{\sone}{SO(2)}

\newcommand{\sub}{\overline{\operatorname{sub}}}

\newcommand{\cont}{\mathcal{C}(\tilde{u}_i, \lambda_0)}

\newcommand{\BIF}{\mathcal{BIF}}
\newcommand{\sphere}{S^{N-1}}

\begin{document}

\title[Sets of solutions]{Structure of sets of solutions of parametrised semi-linear elliptic systems on spheres}%

\author{Anna Go\l\c{e}biewska}
\address{Faculty of Mathematics and Computer Science\\
Nicolaus Copernicus University \\
PL-87-100 Toru\'{n} \\ ul. Chopina $12 \slash 18$ \\
Poland, ORCID 0000--0002--2417--9960}

\email{Anna.Golebiewska@mat.umk.pl}

\author{Piotr Stefaniak}
\address{School of Mathematics, West Pomeranian University of Technology \\PL-70-310 Szcze\-cin, al. Piast\'{o}w $48\slash 49$, Poland, ORCID 0000--0002--6117--2573}

\email{pstefaniak@zut.edu.pl}

\numberwithin{equation}{section}
\allowdisplaybreaks
\date{\today}

\maketitle
\begin{abstract}
In this paper we study a parametrised non-cooperative symmetric semi-linear elliptic system on a sphere. Assuming that there exist critical orbits of the potential, we study the structure of the sets of solutions of the system.
In particular, using the equivariant Rabinowitz Alternative we formulate sufficient conditions for a bifurcation of unbounded sets of solutions.

\end{abstract}

\section{Introduction}

The aim of this paper is to study global bifurcations of weak solutions of a parametrised system on the $(N-1)$-dimensional unit sphere $S^{N-1}$. More precisely, we consider the system
\begin{equation}\label{eq:intro}
\left\{\begin{array}{lclcl}
a_1\Delta_{S^{N-1}} u_1(x)&=&\nabla_{u_1}F(u(x),\lambda) \\
a_2\Delta_{S^{N-1}} u_2(x)&=&\nabla_{u_2}F(u(x),\lambda) \\
\vdots&\vdots&&\ \ \ \text{on}& S^{N-1}, \\
a_p\Delta_{S^{N-1}} u_p(x) &=&\nabla_{u_p}F(u(x),\lambda)
\end{array}
\right.
\end{equation}
where $\Delta_{S^{N-1}}$ is the Laplace--Beltrami operator on $S^{N-1}$,  $a_i\in\{-1,1\}$ for $i=1,\ldots,p$ and $F$ satisfies some additional assumptions.

We consider a situation when the system \eqref{eq:intro} has a family of trivial solutions. For example if  $u_0\in(\nabla_u F(\cdot,\lambda))^{-1}(0)$ for all $\lambda \in \mathbb{R}$, then the constant function $\tilde{u}_0 \equiv u_0$ is a solution of the system for all $\lambda \in \bR$ and we obtain a family of trivial solutions $\{\tilde{u}_0\} \times \bR$. Investigating a change of a topological degree of an operator associated with the system along such a family, one can prove the existence of a continuum (i.e. a closed, connected set) of nontrivial solutions, emanating from this family, i.e. the phenomenon of global bifurcation.

The idea of studying the global bifurcation has been originated by Rabinowitz, see \cite{Rabinowitz}, \cite{Rabinowitz1}, with the use of the Leray-Schauder degree. The same reasoning can be applied if the Leray-Schauder degree is replaced by another invariant, having analogous properties, for example by the degree for invariant strongly indefinite functionals. The counterpart of the Rabinowitz result in this case has been proposed in \cite{GolRyb2011}. Afterwards, this theorem has been applied for studying global bifurcation phenomenon for elliptic systems, for example in \cite{GolKlu}, \cite{GolRyb2011}, \cite{RybSte}.

In particular, the Rabinowitz alternative guarantees that the bifurcating continuum either is unbounded, or it meets the trivial family at another level  $\lambda_0 \in \bR.$  Moreover, in the latter case the continuum  has to satisfy some equality, given in terms of a bifurcation index.  If we exclude this situation, we obtain in this way the existence of unbounded sets of solutions.

Such a method of finding  unbounded continua of solutions has been applied for example in \cite{Ryb1996} and \cite{RybSte} for an elliptic equation and for a system of elliptic equations on spheres. In these papers it has been proved that in the considered problems all bifurcating continua of solutions are unbounded. Similar results for elliptic systems on geodesic balls have been given in \cite{RybShiSte}. We refer to these papers for a discussion of other results about unbounded sets of solutions.

However, the results described above, concerning the global bifurcation phenomenon and the existence of unbounded sets of solutions, are obtained under the assumption that critical points of the potentials of the considered systems are isolated. On the other hand, when there are some additional symmetries, this assumption is usually not satisfied.

In our research we study the systems with symmetries. Note that symmetry of the system \eqref{eq:intro} appears in a natural way, since $S^{N-1}$ is an $SO(2)$-invariant set (by $SO(2)$ we understand the special orthogonal group). However, we assume that also the potential $F$ is invariant with respect to the action of some compact Lie group $\Gamma$.

The symmetries of the system are inherited by the associated  functional. In consequence, the critical points of this functional form orbits and we consider the family of these orbits as a set of trivial solutions. Our aim  is to study global bifurcations from this family. To this end, we use the equivariant version of the Rabinowitz alternative, formulated in \cite{GolSte}. This theorem is expressed in terms of the equivariant bifurcation index, defined in terms of the degree for invariant strongly indefinite functionals. As in the classical version, it states that a bifurcating continuum either is unbounded or it returns to the trivial family. In our paper we study the latter possibility of the alternative and analyse the conditions for the existence of bounded continua. Expressing these conditions in terms of the right-hand side of the system we can easily verify them and, excluding in this way the latter possibility in the Rabinowitz alternative, obtain the existence of unbounded continua of solutions.

The main results of our paper are the theorems concerning  the existence of unbounded sets of solutions of the system \eqref{eq:intro}, namely Theorems \ref{thm:unbound1}, \ref{thm:unboundzero}, \ref{thm:unbounded2},  \ref{thm:unbounded0}. More precisely, we analyse two situations: a system having one critical orbit and a system having two critical orbits. In either case we study the formulae for the sum of bifurcation indices and we examine if the latter possibility of the Rabinowitz alternative can occur. As a result we formulate conditions for the existence of unbounded sets of solutions. In particular, in the case of a system having one critical orbit we obtain that all bifurcating continua are unbounded. In the case of a system having two critical orbits, we obtain such a result under some additional assumptions. Moreover, we prove that without these assumptions, there still have to exist unbounded continua of solutions of \eqref{eq:intro}.

We emphasise that our results are obtained without the assumption that the critical points of the potential are isolated. As far as we know such a situation has not been considered for elliptic systems yet, except in our papers \cite{GolKluSte}, \cite{GRS}, \cite{GolSte}.

\section{Global bifurcation from an orbit}\label{sec:global}
The main tool which we use in our paper is the  equivariant global bifurcation theorem, formulated in \cite{GolSte}. For the convenience of the reader we recall in this section the relevant material concerning this theory.

Throughout this section $G$ denotes a compact Lie group. Note that we use the standard notation concerning representations of Lie groups. In particular, if $\bV$ is a $G$-representation and $v_0 \in \bV$, we denote by $G_{v_0}$ the isotropy group of $v_0$ and by $G(v_0)$ its orbit. We say that a function $\varphi \colon \bV \to \bR$ is $G$-invariant if $\varphi(gv)=\varphi(v)$ for all $g \in G, v \in \bV$ and a function $\varphi\colon \bV \to \bV$ is $G$-equivariant if $\varphi(gv)=g\varphi(v)$ for all $g \in G, v \in \bV$. The detailed exposition of terminology from equivariant topology can be found for example in \cite{Ryb2005milano}.

Let $\bH$ be an infinite dimensional  Hilbert space, which is an orthogonal $G$-representation and assume that there exists an approximation scheme $\{\pi_n\colon \bH \to \bH\colon n \in \bN \cup \{0\}\}$ on $\bH$ (see \cite{GolRyb2011} for the definition of an approximation scheme).

Let $\Phi \in C^2(\bH \times \bR, \bR)$ be a $G$-invariant functional such that $\nabla_u \Phi(u, \lambda)=Lu-\nabla_u \eta(u, \lambda)$, where $L\colon \bH \to \bH$ is a linear, bounded, self-adjoint, $G$-equivariant Fredholm operator of index $0$ satisfying $\pi_0(\bH)=\ker L$ and $\pi_n \circ L = L \circ \pi_n$ for all $n \in \bN \cup \{0\}$, and $\nabla _u \eta\colon \bH \times \bR \to \bH$ is a $G$-equivariant, completely continuous operator.

Assume that there exists $u_0 \in \bH$ such that $\nabla_u \Phi(u_0, \lambda)=0$ for all $\lambda \in \bR$. From the $G$-equivariance of $\nabla_u \Phi$ it follows that $G(u_0)$ is a critical orbit of $\Phi(\cdot,\lambda)$ for every $\lambda\in\bR$, i.e. it consists of critical points of $\Phi(\cdot,\lambda)$. Hence we can consider a family $\cT=G(u_0) \times \bR$ of solutions of $\nabla_u\Phi(u, \lambda)=0.$ The elements of $\cT$ are called the trivial solutions of this equation. On the other hand, we define the family of nontrivial solutions as $\cN=\{(u, \lambda) \in \bH \times \bR \colon \nabla_u \Phi(u, \lambda)=0, (u, \lambda) \notin \cT\}.$ We are interested in studying the phenomenon of bifurcation from the orbit $G(u_0) \times \{\lambda_0\}$ for some $\lambda_0 \in \bR$, i.e. the existence of a sequence of orbits $G(u_n)\times\{\lambda_n\}\subset\cN$ converging to $G(u_0)\times\{\lambda_0\}.$ By such convergence we understand that for all $g\in G$ there holds: $(gu_n,\lambda_n)$ converges to $(gu_0,\lambda_0)$. Note that, by the orthogonality of $\bH$, to obtain a bifurcation  from the orbit $G(u_0)\times\{\lambda_0\}$ it suffices to find a sequence of elements of $\cN$ converging to $(u_0,\lambda_0)$.

It is known (see Fact 2.10 of \cite{GolSte}) that the necessary condition for a bifurcation from $G(u_0)\times\{\lambda_0\}$ is $$\dim \ker\nabla^2_u\Phi(u_0, \lambda_0) > \dim (G(u_0) \times \{\lambda_0\}).$$ Denote by $\Lambda$ the set of $\lambda_0 \in \bR$ satisfying this condition.

In order to formulate the sufficient condition, we introduce a bifurcation index, being an element of the so called Euler ring $U(G)$, see Section \ref{sec:Euler} for a description of this ring. This index is defined with the use of the degree for invariant strongly indefinite functionals, denoted by $\degg(\cdot,\cdot)$, described in \cite{GolRyb2011}, see also Section \ref{sec:degree}.

Fix $\lambda_0 \in \Lambda$ and assume that there exists $\varepsilon >0$ such that $\Lambda \cap [\lambda_0-\varepsilon, \lambda_0+\varepsilon]=\{ \lambda_0\}.$ With this assumption, we obtain that $G(u_0)$ is an isolated critical orbit of $\Phi(\cdot, \lambda_0 \pm \varepsilon)$ and hence there exists  an open, $G$-invariant set $\Omega \subset \bH$ such that $(\nabla_u\Phi(\cdot, \lambda_0 \pm \varepsilon))^{-1}(0) \cap cl(\Omega)=G(u_0)$. Therefore we can define the bifurcation index by
$$\BIF_G(G(u_0),\lambda_0) = \degg(\nabla_u\Phi(\cdot, \lambda_0+\varepsilon), \Omega)-\degg(\nabla_u\Phi(\cdot, \lambda_0 - \varepsilon), \Omega).$$

Note that $\BIF_G(G(u_0),\lambda_0)$ is defined in terms of the degree in a neighbourhood of the orbit. In \cite{GolSte} we have developed the technique of computing such a degree, with the use of the degree on a space normal to the orbit. Using this theory, instead of computing the index $\BIF_G(G(u_0),\lambda_0)$, we can use another index, which is easier to compute. Put $\bW=(T_{u_0}G(u_0))^{\perp}, H=G_{u_0}$ and $\Psi=\Phi_{|\bW}$. Then the bifurcation index, being an element of $U(H)$, given by $$\BIF_H(u_0,\lambda_0) = \degh(\nabla_u\Psi(\cdot, \lambda_0+\varepsilon), \Omega\cap \bW)-\degh(\nabla_u\Psi(\cdot, \lambda_0-\varepsilon), \Omega\cap \bW)$$ is well-defined, see \cite{GolSte} for the details.

We are particularly interested in the case when the pair $(G,H)$ is admissible (i.e. it satisfies the condition that for any closed subgroups $H_1,H_2 \subset H$ there holds the implication: if subgroups $H_1$ and $H_2$ are not conjugate in $H$, then they are not conjugate in $G$, see \cite{PRS} for examples of admissible pairs). In such a situation, to prove the bifurcation from the orbit $G(u_0) \times \{\lambda_0\}$ it is enough to show that $\BIF_H(u_0,\lambda_0) \neq \Theta$, where $\Theta$ is the zero element in the Euler ring $U(H)$. More precisely, from Theorems 2.11 and 2.12  in \cite{GolSte} it follows the bifurcation theorem:

\begin{Theorem}\label{thm:rabinowitz0}
Fix $\lambda_0 \in \Lambda$ and let  the pair $(G,H)$ be admissible. Moreover assume that $\BIF_H(u_0,\lambda_0) \neq \Theta \in U(H)$. Then the bifurcation from the orbit $G(u_0)\times \{\lambda_0\}$ occurs. Furthermore there exists a connected component $\mathcal{C}(u_0,\lambda_0)$ of $cl(\mathcal{N})$, containing $(u_0,\lambda_0).$
\end{Theorem}

If the assertion of the above theorem is satisfied, we obtain a connected component for every $v\in G(u_0)$. In particular, if the group $G$ is connected then there exists a connected subset of $cl(\cN)$ containing  $G(u_0)\times \{\lambda_0\}$.  Therefore, for simplicity  of the exposition from now on we assume that the group $G$ is connected.

In our paper we are interested in studying not only the existence, but also the  behaviour of the continuum $\mathcal{C}(u_0,\lambda_0)$ (by a continuum we understand a closed connected set). It occurs that such a set is either unbounded, or it meets the set $\cT$ at $G(u_0)\times \{\lambda_1\}$ for $\lambda_1 \neq \lambda_0$, see Theorem 2.11 of \cite{GolSte}. Below we formulate a more general version of this result. We consider the set $\cT$ in a more general form, namely, as a union of disjoint sets
of the form $G(u_i) \times \bR$, i.e. there exist $u_1, \ldots, u_q \in \bH$ such that $\nabla_u \Phi(u_i, \lambda)=0$ for all $\lambda \in \bR$ and $i=1, \ldots, q$ and $\cT=(G(u_1) \cup \ldots \cup G(u_q))\times \bR, G(u_i) \cap G(u_j) = \emptyset$ for $i \neq j$. Moreover, we assume that the isotropy groups of  all  elements in $\cT$ are conjugate, i.e. $(G_{u_1})=\ldots=(G_{u_q})$, where by $(H)$ we denote the conjugacy class of $H \subset G$.

In this case, for $u_i \in \{u_1, \ldots, u_q\}$ we denote by $\Lambda_i$ the set of $\lambda_0 \in \bR$ such that  $\dim \ker \nabla^2_u\Phi(u_i, \lambda_0)> \dim (G(u_i) \times\{\lambda_0\})$. Combining the ideas of Theorems 2.11, 2.12 and Remark 2.13 of \cite{GolSte} we obtain the following version of the Rabinowitz alternative:

\begin{Theorem}\label{thm:AltRab}
Fix $ u_i \in \{u_1, \ldots, u_q\}$ and $ \lambda_0\in \Lambda_i$. Assume that $(G_{u_1})=\ldots=(G_{u_q})$  and, denoting $H=G_{u_1}$, that the pair $(G,H)$ is admissible. If $\BIF_H(u_i, \lambda_0) \neq \Theta \in U(H)$, then the continuum  $\mathcal{C}(u_i, \lambda_0)$ bifurcating from $G(u_i) \times \{\lambda_0\}$:
\begin{enumerate}
\item either is unbounded in $\bH \times \bR$
\item or $\mathcal{C}(u_i, \lambda_0) \cap (\cT \setminus (G(u_i) \times \{ \lambda_0\})) \neq \emptyset.$
\end{enumerate}
Moreover, if  (2) occurs, then there exist $r_1, \ldots, r_k \in \bN \cup \{0\}$ such that $\{\lambda_{k,1}, \ldots, \lambda_{k,r_k}\}\subset \Lambda_k $, $\cC(u_i, \lambda_0) \cap \cT=\bigcup_{k=1}^{q} \bigcup_{j=1}^{r_k} G(u_k) \times \{\lambda_{k,j}\}$ and $$\sum_{k=1}^{q} \sum_{j=1}^{r_k} \BIF_H(u_k, \lambda_{k,j})= \Theta. $$
\end{Theorem}

If there exists a continuum satisfying the assertion of the above theorem, we say that the global bifurcation phenomenon occurs from $G(u_i) \times \{\lambda_0\}$. Obviously, the global bifurcation from the orbit implies a bifurcation from this orbit.

\begin{Remark}
Note, that if $\BIF_H(u_i, \lambda_0)\neq \Theta$ and we are able to exclude the possibility (2) of the above theorem, then we obtain the existence of an unbounded set of nontrivial solutions, bifurcating from the orbit $G(u_i) \times \{\lambda_0\}.$
\end{Remark}

\section{Elliptic system with one critical orbit}
In this section we consider a symmetric, parametrised elliptic system with an orbit of constant solutions. We study existence and unboundedness of continua bifurcating from this orbit.

More precisely, we fix a compact and connected Lie group $\Gamma$ and consider the following system of equations:
\begin{equation}\label{eq:system}
\left\{\begin{array}{lclcl}
a_1\Delta_{S^{N-1}} u_1(x)&=&\nabla_{u_1}F(u(x),\lambda) \\
a_2\Delta_{S^{N-1}} u_2(x)&=&\nabla_{u_2}F(u(x),\lambda) \\
\vdots&\vdots&&\ \ \ \text{on}& S^{N-1}, \\
a_p\Delta_{S^{N-1}} u_p(x) &=&\nabla_{u_p}F(u(x),\lambda)
\end{array}
\right.
\end{equation}
where
\begin{enumerate}
\item[(a1)] $F\in C^2(\bR^p\times\bR,\bR)$,
\item[(a2)] $u_0$ is a critical point of $F$ for all $\lambda \in \bR$ and  $\nabla_u^2 F(u_0,\lambda)= \lambda B$, where $B=\diag(b_1,\ldots,b_p)$, $b_i\in\{0,1\}$.
\item[(a3)] $\bR^p$ is an orthogonal $\Gamma$-representation and $F$ is $\Gamma$-invariant with respect to the first variable, i.e.  $F(\gamma u,\lambda) =F(u,\lambda)$ for every $\gamma \in \Gamma$, $u\in\bR^p$, $\lambda\in\bR$.
\item[(a4)] $a_i\in\{-1,1\}$ and $A=\diag(a_1,\ldots,a_p)$ is a $\Gamma$-automorphism, i.e. $A\gamma=\gamma A$ for every $\gamma\in\Gamma$.
\item[(a5)] $\Gamma_{u_0}=\{e\}$.
\end{enumerate}
From (a2) and (a3) it follows that $\nabla_u F(\gamma u_0,\lambda)=0$ for every $\gamma\in \Gamma$, $\lambda\in\bR$, i.e. $\Gamma(u_0)\subset(\nabla_u F(\cdot,\lambda))^{-1}(0)$ for every $\lambda\in\bR$. We additionally assume that:
\begin{enumerate}
\item[(a6)] the orbit $\Gamma(u_0)$ is non-degenerate for every $\lambda\in\bR$, i.e.  $\dim \ker \nabla_u^2 F (u_0,\lambda) = \dim \Gamma (u_0)$.
\end{enumerate}

Note that from the assumption (a6) it follows that the number of $b_i=0$ at the diagonal of $B$ is equal to $\dim \Gamma (u_0)$.

\subsection{Variational setting}\label{subsec:variational}

Let $H^1(S^{N-1})$ denote the Sobolev space with the inner product
\[
\langle \eta,\xi\rangle_{H^1(S^{N-1})}=\int_{S^{N-1}}(\nabla \eta(x), \nabla \xi(x)) +\eta(x) \cdot \xi(x) d\sigma.
\]
Consider the space $\bH=\bigoplus_{i=1}^p H^1(S^{N-1})$ with the inner product given by
\begin{equation*}\label{eq:iloczyn}
\langle u, v\rangle_{\bH} = \sum_{i=1}^p \langle u_i,v_i\rangle_{H^1(S^{N-1})}.
\end{equation*}

Note that we can define an action of the group $G=\Gamma\times SO(2)$ on $\bH$ by $((\gamma,\alpha),u)(x)\mapsto \gamma u(\alpha x)$ for $(\gamma,\alpha)\in G$, $u\in\bH$, $x\in S^{N-1}$, where $\alpha x$ is defined by $(\alpha, (x_1,\ldots,x_N))\mapsto (\alpha (x_1,x_2),x_3,\ldots,x_N)$. Then $\bH$ is an orthogonal $G$-representation and so is the space $\bH \times \bR$ with the action given by $(g,(u,\lambda))\mapsto (gu,\lambda)$ for  $g\in G, u\in \bH,\lambda\in\bR$.

It is known that weak solutions of the system \eqref{eq:system} are in one-to-one correspondence with critical points (with respect to $u$) of the functional $\Phi \colon \bH \times \bR \to \bR$ given by
\begin{equation}\label{Phi}
\Phi(u,\lambda)=\frac{1}{2}\int_{S^{N-1}}\sum^p_{i=1}(-a_i|\nabla u_i(x)|^2)d\sigma-\int_{S^{N-1}}F(u(x),\lambda)d\sigma.
\end{equation}
From the assumptions (a3) and (a4) it follows that $\Phi$ is $G$-invariant.

Denote by $\tilde{u}_0\in\bH$ the constant function $\tilde{u}_0\equiv u_0$  and observe that $G_{\tilde{u}_0}=\{e\} \times SO(2)$.
Analogously as in Lemmas 3.2  and 3.3 of \cite{GolSte} we obtain

\begin{Lemma}\label{lem:postacPhi}
Under the assumptions (a1)--(a4):
\begin{equation*}
\nabla_u\Phi(u,\lambda)=L(u-\tilde{u}_0)+L_{\lambda B} (u-\tilde{u}_0)- \nabla_u\eta_0(u-\tilde{u}_0,\lambda),
\end{equation*}
where
\begin{enumerate}
\item $L\colon\bH\to\bH$ is a linear self-adjoint, bounded Fredholm operator of index 0 given by $L(u_1,\ldots,u_p)=(-a_1u_1,\ldots,-a_pu_p)$,
\item $L_{\lambda B}\colon\bH\to\bH$ is a linear self-adjoint, bounded, completely continuous operator given by
$$\langle L_{\lambda B}u,v\rangle_{\bH}=\int_{\sphere} (Au(x)-\lambda Bu(x),v(x))\, d\sigma$$
for all $v\in \bH$,
\item $\nabla_u\eta_0\colon\bH\times\bR\to\bH$ is a completely continuous operator such that $\nabla_u\eta_0(0,\lambda)=0,\ \nabla^2_u\eta_0(0,\lambda)=0$ for every $\lambda\in\bR$.
\end{enumerate}
\end{Lemma}

Denote by $\sigma(-\Delta_{S^{N-1}})=\{0=\beta_{0}< \beta_1<\beta_2<\ldots\}$ the set of all eigenvalues of the Laplace--Beltrami operator $-\Delta_{S^{N-1}}$ and put $\sigma^-(-\Delta_{S^{N-1}})=\{-\beta_m:\beta_m\in\sigma(-\Delta_{S^{N-1}})\}$.

Let $\cH^N_m$ denote the linear space of harmonic, homogeneous polynomials of $N$ independent variables, of degree $m$, restricted to the sphere $S^{N-1}.$
It is known that $\cH^N_m$ is an eigenspace of $-\Delta_{S^{N-1}}$ with the corresponding eigenvalue $\beta_m=m\cdot(m+N-2)$, $m\in\bN\cup\{0\}$, see \cite{Gurarie}, \cite{Shimakura}.

Denote by $n_-$ (respectively, $n_+$) the number of  coefficients $k$ such that $a_k=-1$, $b_k=1$ (respectively, $a_k=1$ and $b_k=1$). Similarly, by $n_-^0$ (respectively, $n_+^0$) we denote the number of coefficients $k$ such that $a_k=-1$, $b_k=0$ (respectively, $a_k=1$ and $b_k=0$).
Note that $n_-+n_+>0$.

Below we describe the set $\sigma(L+L_{\lambda B})$, i.e. the spectrum of $L+L_{\lambda B}$.

\begin{Lemma}\label{lem:spectrum}
Under the above assumptions:
\begin{align*}
\sigma(L+L_{\lambda B})\subset
\left\{\frac{\beta_m-\lambda}{1+\beta_m}, \frac{-\beta_m-\lambda}{1+\beta_m}, \frac{\beta_m}{1+\beta_m}, \frac{-\beta_m}{1+\beta_m} \colon \beta_m\in \sigma(-\Delta_{\sphere})\right\}. \end{align*}
The multiplicities of the eigenvalues of $L+L_{\lambda B}$ are the following:
\[ \mu\left(\frac{\beta_m-\lambda}{1+\beta_m}\right)=n_-,\ \ \mu\left(\frac{-\beta_m-\lambda}{1+\beta_m}\right)=n_+,\  \  \mu\left(\frac{\beta_m}{1+\beta_m}\right)=n_-^0,\ \ \mu\left(\frac{-\beta_m}{1+\beta_m}\right)=n_+^0. \]
\end{Lemma}

The proof of this lemma is standard, see for example the proof of Lemma 3.2 of \cite{GolKlu}. Note that some of the multiplicities in the above lemma may be zeros.

Define a $G$-equivariant approximation scheme $ \{\pi_n\colon\bH \rightarrow \bH\colon n \in \bN \cup \{0\} \}$ on $\bH$ by
\begin{enumerate}
\item[(b1)] $\bH^0 = \{0\},$
\item[(b2)] $ \bH^n = \bigoplus_{j=1}^{p} \bigoplus_{k=1}^n  \cH^N_{k-1}$,
\item[(b3)] $\pi_n \colon \bH \rightarrow \bH$ is a natural $G$-equivariant projection such that $ \pi_n(\bH) = \bH^n$ for $n \in \bN\cup\{0\}.$
\end{enumerate}
Note that from the definitions of $L$ and $\pi_n$ we have $L \circ \pi_n=\pi_n \circ L$ for all $n \in \bN \cup \{0\}$ and $\ker L = \bH^0$.

\subsection{Global bifurcation}
From the considerations in the previous subsection it follows that $\Phi$ defined by \eqref{Phi} satisfies the assumptions of Section \ref{sec:global}.
In particular, there is a critical orbit $G(\tilde{u}_0)$ of this functional for every $\lambda\in\bR$.
We are going to study the phenomenon of global bifurcation from the set of trivial solutions $\cT=G(\tilde{u}_0)\times\bR$.

Let us denote by $\Lambda$ the set of all $\lambda \in\bR$ such that the orbit $G(\tilde{u}_0)\times\{\lambda\}$ is degenerate, i.e. such that $\dim \ker \nabla^2_u\Phi(\tilde{u}_0, \lambda) >\dim(G(\tilde{u}_0) \times \{\lambda\}).$

In the following lemma we give the necessary condition for a global bifurcation.

\begin{Lemma}\label{lem:nec}
If a bifurcation of solutions of \eqref{eq:system} occurs from the orbit $G(\tilde{u}_0) \times \{\lambda_0\}$, then $\lambda_0 \in \Lambda.$ Moreover,
$$\Lambda=
\left\{
\begin{array}{ll}
\sigma(-\Delta_{S^{N-1}}),& \text{ when } n_->0, n_+=0 \\
\sigma^-(-\Delta_{S^{N-1}}),& \text{ when } n_+>0, n_-=0 \\
\sigma(-\Delta_{S^{N-1}})\cup \sigma^-(-\Delta_{S^{N-1}}),& \text{ when } n_-n_+>0. \\
\end{array}
\right.$$
\end{Lemma}
The proof of this lemma is similar to the one of Lemma 3.1 of \cite{GolKluSte}, so we omit it.

Now we turn our attention to the sufficient condition.  Fix $\lambda_0\in\Lambda$ and note that from the above lemma it follows that the set $\Lambda$ is discrete, therefore we can choose $\varepsilon >0$ such that $\Lambda \cap [\lambda_0 - \varepsilon,\lambda_0 + \varepsilon ]=\{\lambda_0\}$. Hence there exists a $G$-invariant open set $\Omega\subset\bH$ satisfying $(\nabla_u\Phi(\cdot, \lambda_0 \pm \varepsilon))^{-1}(0)\cap \Omega= G(\tilde{u}_0)$. Put $\bW=(T_{\tilde{u}_0} G(\tilde{u}_0))^{\bot}$,  $H=G_{\tilde{u}_0}=\{e\}\times SO(2)$  and $\Psi=\Phi_{|\bW}$. It is known that $\bW$ is an orthogonal $H$-representation. Moreover, the pair $(G,H)$ is admissible, see Lemma 2.1 of \cite{PRS}.

Note that the sequence $\{\tilde{\pi}_n\colon \bW \to \bW \colon n \in \bN \cup \{0\}\}$ of $H$-equivariant orthogonal projections, satisfying $\tilde{\pi}_n(\bW)=\bH^n \cap \bW$, is an $H$-equivariant approximation scheme on $\bW.$ Therefore the degree $\nabla_H\text{-}\deg(\nabla\Psi(\cdot,\lambda_0 \pm \varepsilon), \Omega \cap \bW)$ is well-defined.

To study the global bifurcation phenomenon from the set $\cT$ we are going to apply Theorem \ref{thm:AltRab}. To this end we first investigate the bifurcation index
\begin{multline*}\BIF_{H}(\tilde{u}_0,\lambda_0)= \nabla_{H}\text{-}\deg(\nabla_u\Psi(\cdot,\lambda_0+\varepsilon),\Omega \cap \bW)+\\- \nabla_{H}\text{-}\deg(\nabla_u\Psi(\cdot,\lambda_0-\varepsilon),\Omega \cap \bW)\in U(H).
\end{multline*}
Since  $H=\{e\}\times SO(2)$,  there is an obvious one-to-one correspondence between the indices $\BIF_{H}(\tilde{u}_0,\lambda_0)$ and
\begin{multline}\label{eq:so(2)index}
\BIF_{SO(2)}(\tilde{u}_0,\lambda_0)=\nabla_{SO(2)}\text{-}\deg(\nabla_u\Psi(\cdot,\lambda_0+\varepsilon),\Omega \cap \bW)+\\- \nabla_{SO(2)}\text{-}\deg(\nabla_u\Psi(\cdot,\lambda_0-\varepsilon),\Omega \cap \bW)\in U(SO(2)).
\end{multline}
Here, to consider $\nabla_{SO(2)}\text{-}\deg(\cdot,\cdot)$, we use $\{\tilde{\pi}_n\colon \bW \to \bW \colon n \in \bN \cup \{0\}\}$ as the $SO(2)$-equivariant approximation scheme.
From now on we study the index given by \eqref{eq:so(2)index} for $\lambda_0 \in \Lambda$. To shorten the notation we denote it by $\BIF(\lambda_0)$.

Put $\bV(n)=\bigoplus_{k=0}^n \cH^N_k$ and denote by $B(\bE)$ the open unit ball in a linear space $\bE$.
\begin{Lemma}\label{lem:indices}
Fix $\beta_m\in \sigma(-\Delta_{S^{N-1}})\setminus\{0\}$.
\begin{enumerate}
\item If $n_->0$, then
\[
\BIF(\beta_{m})= \nabla_{SO(2)}\text{-}\deg(-Id, B(\bV(m-1)))^{n_-}\star
\left( \nabla_{SO(2)}\text{-}\deg(-Id, B(\cH^N_m))^{n_-}-\bI \right).
\]
\item If $n_+>0$, then
\[
\BIF(-\beta_{m})=\nabla_{SO(2)}\text{-}\deg(-Id, B(\bV(m)))^{-n_+}\star\left( \nabla_{SO(2)}\text{-}\deg(-Id, B(\cH^N_m))^{n_+}-\bI \right).
\]
\item $ \BIF(0)=\left( (-1)^{n_-}-(-1)^{n_+} \right)\cdot \bI. $
\end{enumerate}
\end{Lemma}

\begin{proof}
Fix $\beta_m\in \sigma(-\Delta_{S^{N-1}})\setminus\{0\}$ and suppose that $n_->0$. We are going to compute the bifurcation index given by \eqref{eq:so(2)index} for $\lambda_0=\beta_m$.
From the definition and the linearisation property of the degree we have
\begin{multline}\label{eq:degreeinproof}
\nabla_{SO(2)}\text{-}\deg(\nabla_u\Psi(\cdot, \beta_{m}\pm\varepsilon), \Omega \cap \bW)=\\=\nabla_{SO(2)}\text{-}\deg(L, B(\bH^n\cap\bW))^{-1} \star
\nabla_{SO(2)}\text{-}\deg(L+ L_{(\beta_m\pm\varepsilon) B}, B(\bH^n\cap\bW)),
\end{multline}
where $n$ is sufficiently large. Taking into consideration the form of L (see Lemma \ref{lem:postacPhi}(1)) and the product formula for  the degree, we have:
\begin{multline*}
\nabla_{SO(2)}\text{-}\deg(L, B(\bH^n\cap\bW))=\\ =\nabla_{SO(2)}\text{-}\deg(-Id, B(\bV(n)))^{n_+}
\star\nabla_{SO(2)}\text{-}\deg(-Id, B(\bV(n)\ominus \bV(0)))^{n_+^0}.
\end{multline*}

Now we are going to compute the latter factor in \eqref{eq:degreeinproof}.
From Lemma \ref{lem:spectrum} it follows that $(L+L_{(\beta_m+\varepsilon)B})_{|\bH^n\cap\bW}$ is homotopic to an operator acting as $-Id$ on
\begin{multline*} \left(\bigoplus_{i=1}^{n_-}\bigoplus_{\substack{\beta_k<\beta_m+\varepsilon\\k=0,\ldots,n}}\cH^N_k\right)\oplus \left(\bigoplus_{i=1}^{n_+}\bigoplus_{\substack{\beta_k>-\beta_m-\varepsilon\\k=0,\ldots,n}}\cH^N_k\right) \oplus\bigoplus_{i=1}^{n_+^0}\left(\bV(n)\ominus \bV(0)\right)=\\=\bigoplus_{i=1}^{n_-}\bV(m)\oplus \bigoplus_{i=1}^{n_+}\bV(n)\oplus\bigoplus_{i=1}^{n_+^0}\left(\bV(n)\ominus \bV(0)\right)
\end{multline*}
and as $Id$ on the orthogonal complement of this space. Since $\nabla_{SO(2)}\text{-}\deg(Id, B(\bX))=\bI$ for any $SO(2)$-representation $\bX$,  applying the homotopy property and the product formula for the degree, we get
\begin{align*}
\nabla_{SO(2)}\text{-}\deg(L+L_{(\beta_m+\varepsilon) B}, B(\bH^n\cap\bW))
 =\nabla_{SO(2)}\text{-}\deg(-Id, B(\bV(m)))^{n_-}\star  \\ \nabla_{SO(2)}\text{-}\deg(-Id, B(\bV(n))))^{n_+}
 \star \nabla_{SO(2)}\text{-}\deg(-Id, B(\bV(n)\ominus \bV(0)))^{n_+^0}.
\end{align*}
Analogously
\begin{align*}
 \nabla_{SO(2)}\text{-}\deg(L+L_{(\beta_m-\varepsilon) B}, B(\bH^n\cap\bW))=\nabla_{SO(2)}\text{-}\deg(-Id, B(\bV(m-1)))^{n_-}\star \\ \star \nabla_{SO(2)}\text{-}\deg(-Id, B(\bV(n)))^{n_+}\star \nabla_{SO(2)}\text{-}\deg(-Id, B(\bV(n)\ominus \bV(0)))^{n_+^0}.
\end{align*}
Hence
\begin{align*}
\BIF(\beta_{m})=& \nabla_{SO(2)}\text{-}\deg(\nabla_u\Psi(\cdot, \beta_m+\varepsilon),\Omega \cap \bW)- \nabla_{SO(2)}\text{-}\deg(\nabla_u\Psi(\cdot, \beta_m-\varepsilon),\Omega \cap \bW)=\\
=&\nabla_{SO(2)}\text{-}\deg(-Id, B(\bV(n)))^{-n_+}\star \nabla_{SO(2)}\text{-}\deg(-Id, B(\bV(n)\ominus \bV(0)))^{-n_+^0}\star\\
\star & \nabla_{SO(2)}\text{-}\deg(-Id, B(\bV(m-1)))^{n_-}
\star \nabla_{SO(2)}\text{-}\deg(-Id, B(\bV(n)))^{n_+}\star \\
\star & \nabla_{SO(2)}\text{-}\deg(-Id, B(\bV(n)\ominus \bV(0)))^{n_+^0} \star
\left( \nabla_{SO(2)}\text{-}\deg(-Id, B(\cH^N_m)))^{n_-}-\bI \right) =\\
= & \nabla_{SO(2)}\text{-}\deg(-Id, B(\bV(m-1)))^{n_-}\star
\left( \nabla_{SO(2)}\text{-}\deg(-Id, B(\cH^N_m))^{n_-}-\bI \right).
\end{align*}
The proof of the remaining two equalities is analogous.
\end{proof}

Recall that the coordinates of the $SO(2)$-degree and $SO(2)$-bifurcation index can be written as  $(\alpha_{SO(2)}, \alpha_{\bZ_1}, \alpha_{\bZ_2}\ldots)\in U(SO(2))$, see Appendix.
Using the above lemma, we can compute the precise values of the coordinates of the bifurcation indices. To this end we use the descriptions of the spaces $\cH^N_m$ and $\bV(m)$ as $SO(2)$-representations, given in Section \ref{sec:representations}. Consider the numbers $k_j^m$ and $r_j^m$ given in these descriptions.

\begin{Lemma}\label{lem:coordinates}
Fix $\beta_m\in \sigma(-\Delta_{S^{N-1}})\setminus\{0\}$.
\begin{enumerate}
\item If $n_->0$, then for every $l\leq m$
\[\BIF(\beta_{m})_{\bZ_{l}}= (-1)^{ n_-\dim\bV(m)}n_- \big((-1)^{n_-\dim \cH^N_m}r^{m-1}_l-r^{m-1}_l-k^m_l \big)\]
 and for every $l> m$
 \[ \BIF(\beta_{m})_{\bZ_{l}}= 0.\]
 Moreover,
\[
\BIF(\beta_{m})_{SO(2)}=(-1)^{n_-\dim\bV(m)}\left((-1)^{n_-\dim \cH^N_m}-1\right). \]
\item If $n_+>0$, then for every $l\leq m$
 \[\BIF(-\beta_{m})_{\bZ_{l}}= (-1)^{ n_+\dim\bV(m)}n_+ \big((-1)^{n_+\dim \cH^N_m}r^{m-1}_l-r^{m-1}_l-k^m_l \big)\]
 and for every $l> m$
 \[
\BIF(-\beta_{m})_{\bZ_{l}}= 0.
\]
Moreover,
\[ \BIF(-\beta_{m})_{SO(2)}=(-1)^{n_+\dim\bV(m)}\left((-1)^{n_+\dim \cH^N_m}-1\right).
\]
\end{enumerate}
\end{Lemma}
\begin{proof}
The assertion follows from the formula for the degree of $-Id$ on a given $SO(2)$-representation and the multiplication formula in $U(SO(2))$. More precisely, from \eqref{eq:rozklad}, \eqref{eq:v(m)}, \eqref{eq:stopienId} in  Appendix, we have, for $l\leq m$,
\begin{align*}
\nabla_{SO(2)}\text{-}\deg_{SO(2)}(-Id, B(\bV(m-1)))&=(-1)^{\dim\bV(m-1)},\\
\nabla_{SO(2)}\text{-}\deg_{\bZ_{l}}(-Id, B(\bV(m-1)))&=(-1)^{\dim\bV(m-1)+1} \cdot r^{m-1}_l,\\
\nabla_{SO(2)}\text{-}\deg_{SO(2)}(-Id, B(\cH^N_m))&=(-1)^{\dim \cH^N_m}, \\
\nabla_{SO(2)}\text{-}\deg_{\bZ_{l}}(-Id, B(\cH^N_m))&=(-1)^{\dim \cH^N_m+1}\cdot k^m_l.
\end{align*}
Therefore, by Lemma \ref{lem:indices} and the formula \eqref{eq:powersinU(SO(2))},
\begin{eqnarray*}
&&\BIF(\beta_{m})_{\bZ_{l}}=(-1)^{(n_--1)\dim\bV(m-1) }n_-
 \big( (-1)^{\dim\bV(m-1)+1}  r^{m-1}_l \cdot (-1)^{n_-\dim \cH^N_m} + \\
 && - (-1)^{\dim\bV(m-1)+1} r^{m-1}_l +
  (-1)^{\dim\bV(m-1)}(-1)^{(n_--1)\dim \cH^N_m} (-1)^{\dim \cH^N_m +1} k^m_l
\big)= \\
&& =(-1)^{(n_--1)\dim\bV(m-1) }n_- (-1)^{\dim\bV(m-1) }
\big((-1)^{n_-\dim \cH^N_m +1}r^{m-1}_l+r^{m-1}_l-(-1)^{n_-\dim \cH^N_m}k^m_l \big) =\\
&&= (-1)^{n_-\dim\bV(m-1) }n_- \big((-1)^{n_-\dim \cH^N_m +1}r^{m-1}_l+r^{m-1}_l-(-1)^{n_-\dim \cH^N_m}k^m_l \big)
=\\
&&= (-1)^{ n_-\dim\bV(m)}n_- \big((-1)^{n_-\dim \cH^N_m}r^{m-1}_l-r^{m-1}_l-k^m_l \big).
\end{eqnarray*}
The proof of the remaining equalities is analogous.
\end{proof}

\begin{Corollary}\label{cor:coordinates}
From the above lemma, since $r^{m-1}_m=0$, $k^m_m=1$, for every $m>0$ it follows that
\begin{equation}\label{eq:mthcoordinate}
\BIF(\beta_{m})_{\bZ_{m}}=n_-(-1)^{n_-\dim\bV(m)+1}, \quad \BIF(-\beta_{m})_{\bZ_{m}}= n_+(-1)^{n_+\dim\bV(m)+1}.
\end{equation}
In particular, for every $m>0$:
\begin{enumerate}[(i)]
\item if $n_{-}>0$, then $\BIF(\beta_{m})\neq\Theta$. Therefore a global bifurcation occurs from the orbit $G(\tilde{u}_0)\times\{\beta_m\}$,
\item if $n_{+}>0$, then $\BIF(-\beta_{m})\neq\Theta$.  Therefore a global bifurcation occurs from the orbit $G(\tilde{u}_0)\times\{-\beta_m\}$.
\end{enumerate}
\end{Corollary}

Denote by $\cC(\tilde{u}_0,\lambda_0)$ the connected component of the set
\[
cl\{(u, \lambda) \in \bH \times \bR \colon \nabla_u \Phi(u, \lambda)=0, (u, \lambda) \notin \cT\},
\]
containing $G(\tilde{u}_0)\times\{\lambda_0\}$.

\begin{Theorem}\label{thm:unbound1}
Consider the system \eqref{eq:system} with the potential $F$ and $u_0$ satisfying the assumptions (a1)-(a6) and fix $\beta_{m_0}\in \sigma(-\Delta_{S^{N-1}})\setminus\{0\}$.
\begin{enumerate}
\item If $n_->0$, then the continuum $\cC(\tilde{u}_0,\beta_{m_0})\subset\bH\times\bR$ of weak solutions of the system \eqref{eq:system} is unbounded.
\item If $n_+>0$, then the continuum $\cC(\tilde{u}_0,-\beta_{m_0})\subset\bH\times\bR$ of weak solutions of the system \eqref{eq:system} is unbounded.
\end{enumerate}
\end{Theorem}
\begin{proof}
We prove only (1), the proof of (2) is analogous.

Let $n_->0$. From Corollary \ref{cor:coordinates} there exists a continuum  $\mathcal{C}(\tilde{u}_0, \beta_{m_0})$ bifurcating from $G(\tilde{u}_0) \times \{\beta_{m_0}\}$. We will prove that this continuum is unbounded.

Suppose, by contrary, that $\mathcal{C}(\tilde{u}_0, \beta_{m_0})$ is bounded. Then, by Theorem \ref{thm:AltRab},
there exist $\lambda_1,\ldots,\lambda_r\in\Lambda$ such that $\cC(\tilde{u}_0, \lambda_0) \cap \cT=\bigcup_{j=1}^{r} G(\tilde{u}_0) \times \{\lambda_{j}\}$ and
\begin{equation}\label{eq:sumofindices}
\sum_{j=1}^{r} \BIF(\lambda_{j})= \Theta\in U(SO(2)).
\end{equation}

To prove that the continuum is unbounded, we will show that the equality \eqref{eq:sumofindices} cannot be satisfied.
More precisely, it cannot be satisfied at the $\bZ_{\mu}$-th coordinate, where $\mu$ is such that  $\beta_{\mu}=\max\{|\lambda_{1}|,\ldots,|\lambda_{r}|\}$.
From Lemma \ref{lem:coordinates} it follows that for every $|\lambda_j|<\beta_{\mu}$, $j=1,\ldots,r$, $\BIF(\lambda_j)_{\bZ_{\mu}}=0$. Therefore, the only nonzero summands in $\sum_{j=1}^{r} \BIF(\lambda_{j})_{\bZ_{\mu}}$ can be the ones corresponding to $|\lambda_j|=\beta_{\mu}$, which can be written as
\begin{equation*}
\sum_{j=1}^{r} \BIF(\lambda_{j})_{\bZ_{\mu}}= \alpha_{\mu}\BIF(\beta_{\mu})_{\bZ_{\mu}}+\alpha_{-\mu}\BIF(-\beta_{\mu})_{\bZ_{\mu}},
\end{equation*}
where $\alpha_{\pm \mu}\in\{0,1\}$.
From  \eqref{eq:mthcoordinate} and \eqref{eq:sumofindices} we get
\begin{equation}\label{eq:niewiemjak}
\alpha_{\mu}n_-(-1)^{n_-\dim\bV({\mu})+1}+\alpha_{-{\mu}}n_+(-1)^{n_+\dim\bV({\mu})+1}=0.
\end{equation}
We will prove that this equality cannot be satisfied for any values $\alpha_{\pm\mu}$ and $n_{\pm}$.
If $n_+=0$, then, by Lemma \ref{lem:nec}, $\alpha_{-\mu}=0$ and  \eqref{eq:niewiemjak} is equivalent to $\alpha_{\mu}n_-=0$. From the definition of $\mu$, in this case $\alpha_{\mu}\neq0$ and hence we obtain $n_-=0$, which contradicts $n_->0$.

Therefore $n_{+}>0$ and the equation \eqref{eq:niewiemjak} can be satisfied only if $\alpha_{\pm\mu}>0$, i.e. $\alpha_{\mu}=\alpha_{-\mu}=1$.
If $n_-$ and $n_+$ are of the same parity, then $\alpha_{\mu}n_-+\alpha_{-{\mu}}n_+>0$, which contradicts \eqref{eq:niewiemjak}.
If $n_-$ and $n_+$ are of different parity, then from \eqref{eq:niewiemjak} either $\alpha_{\mu}n_-+\alpha_{-{\mu}}n_+=0$ or $\alpha_{\mu}n_--\alpha_{-{\mu}}n_+=0$. Reasoning similarly as before we obtain a contradiction. This finishes the proof.
\end{proof}

In the previous theorem we have considered all possible levels of bifurcation, except $\lambda_0=0$. Now we are going to study the global bifurcation from the orbit $G(\tilde{u}_0) \times \{0\}$.

\begin{Theorem}\label{thm:unboundzero}
Consider the system \eqref {eq:system} with the potential $F$ and $u_0$ satisfying the assumptions (a1)-(a6). If $n_-, n_+$ are of different parity, then the continuum $\cC(\tilde{u}_0,0)\subset\bH\times\bR$ of weak solutions of the system \eqref{eq:system} is unbounded.
\end{Theorem}

\begin{proof}
From Lemma \ref{lem:indices}(3) it follows that $\BIF(0)\neq\Theta\in U(SO(2))$, therefore, by Theorem \ref{thm:rabinowitz0}, there exists a continuum  $\mathcal{C}(\tilde{u}_0, 0)$ bifurcating from $G(\tilde{u}_0) \times \{0\}$. If the continuum $\mathcal{C}(\tilde{u}_0, 0)$ is bounded, then, by Theorem \ref{thm:AltRab}, it contains $G(\tilde{u}_0)\times\{\lambda_{m_0}\}$ for some $\lambda_{m_0}\in \Lambda\setminus\{0\}$. On the other hand, from the above theorem it follows that the continuum bifurcating from $G(\tilde{u}_0)\times\{\lambda_{m_0}\}$ is unbounded and hence so must be $\mathcal{C}(\tilde{u}_0, 0)$. This completes the proof.
\end{proof}

\begin{Remark}
The special case of the above theorems is the situation when $\Gamma=\{e\}$, i.e. the critical orbit of the potential $F$ consists of one element. This case has been considered in \cite{RybSte}. The above theorems generalise the result from this paper.
\end{Remark}

\section{Elliptic system with two critical orbits}

In the previous section we have considered the elliptic system \eqref{eq:system} with the potential $F$ having one critical orbit $\Gamma(u_0)$. This implied the existence of the trivial family having one connected component. Now we turn our attention to the situation with more than one critical orbit of $F$, i.e. with a trivial family with more than one component. The Rabinowitz alternative states that a bifurcating continuum can be unbounded or it can return to the trivial family, connecting different orbits. Thus, if we are able to exclude the latter possibility, we can prove unboundedness of such continua.

For simplicity of computations we describe only the case of two critical orbits.
More precisely, we consider as before a compact and connected Lie group $\Gamma$ and we study the system \eqref{eq:system} with the assumptions (a1), (a3)-(a5), replacing (a2) by
\begin{enumerate}
\item[(a2')]  $u_1, u_2$ are critical points of $F$ for all $\lambda \in \bR$, where $\Gamma(u_1)\cap\Gamma(u_2)=\emptyset$ and $\nabla^2 F(u_1,\lambda)= \nabla^2 F(u_2,\lambda) = \lambda B$, where $B=\diag(b_1,\ldots,b_p)$, $b_i\in\{0,1\}$.
\end{enumerate}
From the above assumption and (a3) it follows that $\Gamma(u_1)\cup\Gamma(u_2)\subset(\nabla_u F(\cdot,\lambda))^{-1}(0)$ for every $\lambda\in\bR$. We assume that:
\begin{enumerate}
\item[(a6')] the orbits $\Gamma(u_1)$, $\Gamma(u_2)$ are non-degenerate for every $\lambda\in\bR$.
\end{enumerate}

As before, we consider weak solutions of the system as critical points of the functional $\Phi$ defined by \eqref{Phi}. Since the problem is invariant, from our assumptions it follows that $G(\tilde{u}_1)\cup G(\tilde{u}_2)\subset(\nabla_u \Phi(\cdot,\lambda))^{-1}(0)$ for every $\lambda\in\bR$, where $\tilde{u}_1$ and $\tilde{u}_2$ are the constant functions defined by $\tilde{u}_1\equiv u_1$, $\tilde{u}_2\equiv u_2$. Hence we have a family of weak solutions of the system \eqref{eq:system}, i.e. the family $\cT=(G(\tilde{u}_1)\cup G(\tilde{u}_2))\times\bR$. We call the elements of $\cT$ the trivial solutions of the problem.

We are going to study the global bifurcation phenomenon from $\cT$.
As before, we obtain that this phenomenon can occur only at $G(\tilde{u}_i)\times\{\lambda_0\}$, where $i=1,2$, $\lambda_0\in \Lambda$ and $\Lambda$ is given in Lemma \ref{lem:nec}.
 To study the bifurcation problem we use the bifurcation indices defined by \eqref{eq:so(2)index} for $u_0 \in \{u_1,u_2\}, \lambda_0 \in \Lambda$. Repeating the reasoning of the proof of Lemma \ref{lem:indices} and taking into consideration the assumption (a2'), we obtain that  $\BIF_{SO(2)}(\tilde{u}_1,\lambda_0)=\BIF_{SO(2)}(\tilde{u}_2,\lambda_0)$ for every $\lambda_0 \in \Lambda$. To shorten the notation we denote these indices by $\BIF(\lambda_0)$. Note that Lemmas \ref{lem:indices} and \ref{lem:coordinates}  give formulae for $\BIF(\lambda_0)$.

Our aim is to study  the sets of solutions of the system. Reasoning as in the previous section, we first describe levels at which the global bifurcation occurs.

\begin{Lemma}\label{lem:existenceofacontinuum}
Fix $\beta_m\in \sigma(-\Delta_{S^{N-1}})\setminus\{0\}$.
\begin{enumerate}
\item If $n_->0$, then the global bifurcation phenomenon occurs at $G(\tilde{u}_i)\times\{\beta_m\}$, $i=1,2$.
\item If $n_+>0$, then the global bifurcation phenomenon occurs at $G(\tilde{u}_i)\times\{-\beta_m\}$, $i=1,2$.
\item  If $n_-, n_+$ are of different parity, then the global bifurcation phenomenon occurs at $G(\tilde{u}_i)\times\{0\}$, $i=1,2$.
\end{enumerate}
\end{Lemma}
\begin{proof}
Taking into consideration the formulae for bifurcation indices given by Lemma \ref{lem:coordinates} we obtain that in each case $\BIF(\beta_{m})\neq\Theta\in U(SO(2))$. Therefore the assertion follows from Theorem \ref{thm:AltRab}.
\end{proof}

\begin{Remark} Summing up, the  assertions (1) and (2) of the above lemma imply that for all $\lambda_0 \in \Lambda \setminus \{0\}$ the global bifurcation phenomenon occurs at $G(\tilde{u}_i) \times \{ \lambda_0\}.$ Therefore, for $\lambda \neq 0$, the necessary condition of global bifurcation is also a sufficient one.
\end{Remark}

In the rest of this section we study the structure of continua bifurcating from the family $\cT$. Denote by $\cC(\tilde{u}_i, \lambda_0)$ the continuum bifurcating from $G(\tilde{u}_i) \times \{\lambda_0\}$ for $\lambda_0 \in \Lambda$. From the above considerations it follows that we can study only continua of this form. 
We consider two possibilities of the Rabinowitz alternative for these continua and study under which conditions they can occur.
More precisely, Theorem \ref{thm:AltRab} implies that $\cC(\tilde{u}_i, \lambda_0)$ either is unbounded or it intersects $\cT$ at a level $\lambda \neq \lambda_0$, (by such an intersection we understand that $\cC(\tilde{u}_i, \lambda_0) \cap (G(\tilde{u}_1)\times\{\lambda\})\neq \emptyset$ or $\cC(\tilde{u}_i, \lambda_0)\cap (G(\tilde{u}_2)\times\{\lambda\})\neq \emptyset$). In the case of bounded continua, we are interested in describing levels at which such intersections can occur. Roughly speaking, in this way we obtain continua connecting orbits at different levels.

Fix $\mathcal{C}(\tilde{u}_i, \lambda_0)$, $\lambda_0\in\Lambda\setminus \{0\}$.  For this continuum  we denote by $\alpha_{\pm l} \in \{0,1,2\}$ the number of intersections of $\mathcal{C}(\tilde{u}_i, \lambda_0)$ with $\cT$ at the level $\pm\beta_l$, i.e. the number of $G$-orbits in $\mathcal{C}(\tilde{u}_i, \lambda_0) \cap((G(\tilde{u}_1)\times\{\pm\beta_l\})\cup (G(\tilde{u}_2)\times\{\pm\beta_l\})).$
If $\mathcal{C}(\tilde{u}_i, \lambda_0)$ is bounded, we put
\begin{equation}\label{eq:mu}
\mu=\max\{l:\alpha_l+\alpha_{-l} >0\}.
\end{equation}
In other words, $\beta_{\mu}$ is the maximal or $-\beta_{\mu}$ is the minimal  level at $\cT$ met by the continuum, i.e. if $\cT$ is intersected by $\mathcal{C}(\tilde{u}_i, \lambda_0) $ at a level $\lambda_1$ then $|\lambda_1|\leq \beta_{\mu}$.

\begin{Theorem}\label{thm:unbounded1}
 Consider the system \eqref {eq:system} with the potential $F$ and $u_1, u_2$ satisfying the assumptions (a1), (a2'), (a3)-(a5) and (a6') and fix $\lambda_0\in\Lambda\setminus\{0\}$.
 If one of the following conditions is satisfied:
\begin{itemize}
\item[1)] $n_+,n_-$ are of the same parity,
\item[2)] $n_-\neq 2 n_+$ and $n_+\neq 2 n_-,$
\end{itemize}
then, for $i=1,2$, the continuum $\mathcal{C}(\tilde{u}_i, \lambda_0)$ is unbounded.
\end{Theorem}

\begin{proof}
Fix $\lambda_0\neq 0$ and suppose that $\mathcal{C}(\tilde{u}_i, \lambda_0)$ is bounded. Then, by Theorem \ref{thm:AltRab}, there exist $\lambda_{1,1},\ldots,\lambda_{1,r_1},\lambda_{2,1},\ldots,\lambda_{2,r_2}\in\Lambda$ such that
$$\mathcal{C}(\tilde{u}_i, \lambda_0)\cap\cT=G(\tilde{u}_1)\times\{\lambda_{1,1},\ldots,\lambda_{1,r_1}\}\cup G(\tilde{u}_2)\times\{\lambda_{2,1},\ldots,\lambda_{2,r_2}\}$$
and
\[
\textbf{}\sum_{i=1}^{2} \sum_{j=1}^{r_i} \BIF_{SO(2)}(\tilde{u}_i, \lambda_{i,j})=\Theta\in U(SO(2)).
\]
Taking into consideration the form of elements of $U(SO(2))$, from the above equality and Lemma \ref{lem:coordinates}, we obtain
\begin{eqnarray*}
0=\sum_{i=1}^{2} \sum_{j=1}^{r_i} \BIF_{SO(2)}(\tilde{u}_i, \lambda_{i,j})_{\bZ_{\mu}}&= &\alpha_{\mu}\BIF(\beta_{\mu})_{\bZ_{\mu}}+\alpha_{-\mu}\BIF(-\beta_{\mu})_{\bZ_{\mu}}=\\
& =&\alpha_{\mu}n_-(-1)^{n_-\dim\bV(\mu)+1}+\alpha_{-\mu}n_+(-1)^{n_+\dim\bV(\mu)+1},
\end{eqnarray*}
where $\mu$ is given by \eqref{eq:mu}.
If $n_-=0$, then, by Lemma \ref{lem:nec}, 
$\alpha_{\mu}=0$ and therefore, from the above equality, $\alpha_{-\mu}n_+=0$. From the definition of $\mu$, in this case $\alpha_{-\mu}\neq0$ and therefore $n_+=0$, which is impossible and hence the continuum is unbounded. In the same way we prove the unboundedness of the continuum for $n_+=0$.

Suppose now that $n_+,n_-\neq 0$.
If $n_+,n_-$ are of the same parity, then $\alpha_{\mu}n_-+\alpha_{-\mu}n_+=0$, which is not possible.
Hence $n_+,n_-$ are of different parity and therefore $\alpha_{\mu}n_-=\alpha_{-\mu}n_+$, which is possible only if $n_-= 2 n_+$ or $n_+= 2 n_-$. This completes the proof.
\end{proof}

In particular, the above theorem allows to describe bifurcating continua for the cooperative system, i.e. the system satisfying $n_+ \cdot n_-=0$. We give this description in the corollary below.

\begin{Corollary}\label{cor:cooperative}
Suppose that $n_- >0$ and $n_+=0$. From Lemma \ref{lem:existenceofacontinuum} it follows that a continuum bifurcates from $G(\tilde{u}_i)\times\{\beta_m\}$, for $i=1,2$ and $\beta_m\in \sigma(-\Delta_{S^{N-1}})\setminus\{0\}$. From Theorem \ref{thm:unbounded1} it is unbounded. If furthermore $n_-$ is odd, then reasoning similarly as in the proof of Theorem \ref{thm:unboundzero} we obtain that an unbounded continuum bifurcates also from $G(\tilde{u}_i)\times\{0\}$, for $i=1,2$.

Similarly, if $n_- =0$ and $n_+>0$, then an unbounded continuum bifurcates from $G(\tilde{u}_i)\times\{-\beta_m\}$, for $i=1,2$ and $\beta_m\in \sigma(-\Delta_{S^{N-1}})\setminus\{0\}$. If furthermore $n_+$ is odd, then an unbounded continuum bifurcates from $G(\tilde{u}_i)\times\{0\}$, for $i=1,2$.
\end{Corollary}

In Theorem \ref{thm:unbounded1} we have proved that under additional assumptions all bifurcating continua are unbounded. In the theorem below we consider a more general situation without these assumptions and we prove that also in this case there exist unbounded continua.

\begin{Theorem}\label{thm:unbounded2}
 Consider the system \eqref {eq:system} with the potential $F$ and $u_1, u_2$ satisfying the assumptions (a1), (a2'), (a3)-(a5) and (a6') and fix $\beta_m\in \sigma(-\Delta_{S^{N-1}})\setminus\{0\}$. Then for all $m \in \bN$ at least one of the four continua $\cC(\tilde{u}_i,\pm\beta_{m})$, $i=1,2$, is unbounded.
\end{Theorem}

\begin{proof}
Fix $\beta_m\in \sigma(-\Delta_{S^{N-1}})\setminus\{0\}$. From Theorem \ref{thm:unbounded1} it follows that if $n_+,n_-$ are of the same parity or $n_-\neq 2 n_+$ and $n_+\neq 2 n_-$, then all the continua $\cC(\tilde{u}_i,\pm\beta_{m})$, $i=1,2$, are unbounded and therefore the assertion is proved.

If the assumptions of Theorem \ref{thm:unbounded1} on $n_+,n_-$ are not satisfied, i.e. one of the conditions holds:
\begin{enumerate}
\item[(C1)] $n_+$ is odd and $n_-=2n_+$,
\item[(C2)] $n_-$ is odd and $n_+=2n_-$,
\end{enumerate}
then it may happen that some continua of solutions are bounded. To prove the assertion in this case we will describe the structure of bounded continua, namely, we will prove that if $\mathcal{C}(\tilde{u}_i, \lambda_0)$ is bounded, $\lambda_0\in\Lambda\setminus\{0\}$, then for each $l>0$:
\begin{enumerate}
\item[(A1)]  in the case (C1) if $\alpha_l+\alpha_{-l}>0$, then either $\alpha_l=1$, $\alpha_{-l}=2$ or $\alpha_l=2$, $\alpha_{-l}=0$,
\item[(A2)] in the case (C2) if $\alpha_l+\alpha_{-l}>0$, then either $\alpha_l=2$, $\alpha_{-l}=1$ or $\alpha_l=0$, $\alpha_{-l}=2$,
\end{enumerate}
where $\alpha_{\pm l} \in \{0,1,2\}$ is, as before, the number of intersections of $\mathcal{C}(\tilde{u}_i, \lambda_0)$ with $\cT$ at the level $\pm\beta_l$.

We will prove (A1), the analysis in (A2) is analogous. Therefore we assume that $n_+$ is odd, $n_-=2n_+$ and $\cont$ is bounded. Then, by Theorem \ref{thm:AltRab}, there exist $\lambda_{1,1},\ldots,\lambda_{1,r_1},\lambda_{2,1},\ldots,\lambda_{2,r_2}\in\Lambda$ such that
\begin{equation}\label{eq:gwiazdka}\mathcal{C}(\tilde{u}_i, \lambda_0)\cap\cT=G(\tilde{u}_1)\times\{\lambda_{1,1},\ldots,\lambda_{1,r_1}\}\cup G(\tilde{u}_2)\times\{\lambda_{2,1},\ldots,\lambda_{2,r_2}\}
\end{equation}
and
\[
\textbf{}\sum_{i=1}^{2} \sum_{j=1}^{r_i} \BIF_{SO(2)}(\tilde{u}_i, \lambda_{i,j})=\Theta\in U(SO(2)).
\]

Let $\mu$ be defined for $\cont$ by \eqref{eq:mu}. Moreover, recall that $\BIF_{SO(2)}(\tilde{u}_1,\lambda_0)=\BIF_{SO(2)}(\tilde{u}_2,\lambda_0)$ for every $\lambda_0 \in \Lambda$ and we denote this value by $\BIF(\lambda_0)$.  Then the above equality is equivalent to
\begin{equation*}
\alpha_{0}\BIF(0)+\sum_{k=1}^{\mu} \left(\alpha_{k}\BIF(\beta_k)+\alpha_{-k}\BIF(-\beta_k)\right)=\Theta\in U(SO(2)).
\end{equation*}
Taking into consideration the form of elements of $U(SO(2))$, from the above equality we conclude that for all $l \in \bN$
\begin{equation}\label{eq:sumcoord}
\alpha_{0}\BIF(0)_{\bZ_l}+\sum_{k=1}^{\mu} \left(\alpha_{k}\BIF(\beta_k)_{\bZ_l}+\alpha_{-k}\BIF(-\beta_k)_{\bZ_l}\right)=0.
\end{equation}

Note that since $\mu$ is defined by \eqref{eq:mu}, it has to be such that $\dim \bV(\mu)$ is odd.
Indeed, assume that $\dim \bV(\mu)$ is even. Then from \eqref{eq:sumcoord} (taken for $l=\mu$) and Lemma \ref{lem:coordinates} we obtain
\begin{equation}\label{eq:mthcoord}
\alpha_{\mu}\BIF(\beta_{\mu})_{\bZ_{{\mu}}}+\alpha_{-{\mu}}\BIF(-\beta_{\mu})_{\bZ_{{\mu}}}= 0.
\end{equation}
Since, by Corollary \ref{cor:coordinates},
\[
\BIF(\beta_{{\mu}})_{\bZ_{{\mu}}}=-n_-, \quad \BIF(-\beta_{{\mu}})_{\bZ_{{\mu}}}=-n_+,
\]
we have
\[
\alpha_{\mu}\BIF(\beta_{\mu})_{\bZ_{{\mu}}}+\alpha_{-{\mu}}\BIF(-\beta_{\mu})_{\bZ_{{\mu}}}= -(\alpha_{\mu}n_-+\alpha_{-{\mu}}n_+)< 0,\]
which contradicts \eqref{eq:mthcoord} and proves that $\dim\bV({\mu})$ must be odd.
In particular, if $\dim \cH^N_1$ is odd, then $\mu$ cannot be equal to 1.

Now our idea is to describe which values $\pm \beta_l, l \leq \mu$ occur  as the values of $\lambda_{i,j}$  in \eqref{eq:gwiazdka}.
 We use the following scheme: we assume that  for some $M\in\{0,\ldots,\mu\}$ the continuum $\cont$ intersects $\cT$ at levels  $\beta_l$ or $-\beta_l$ for all $M \leq l \leq \mu$. Then we study the sum
\begin{equation}\label{eq:sumM}
\sum_{l=M}^{\mu} \alpha_l \BIF(\beta_l) +\alpha_{-l}\BIF(-\beta_l),
\end{equation}
describing two cases:
\begin{enumerate}[(a)]
\item the sum \eqref{eq:sumM} is equal to $\Theta \in U(SO(2))$,
\item the sum \eqref{eq:sumM} is not equal to $\Theta \in U(SO(2))$. \end{enumerate}
In the case (a) comparing \eqref{eq:sumM} with \eqref{eq:sumcoord} we obtain that  two possibilities can occur:
\begin{enumerate}[(i)]
\item  either the levels $\beta_l$ or $-\beta_l$ for  all $M \leq l \leq \mu$ are all the levels of intersection of $\cont$ with $\cT$
\item or $\cont$ intersects $\cT$ at some other levels $\pm \beta_k$, where $k<M$. In this case we repeat the reasoning replacing $\mu$ by $\mu_1=\max\{l<M: \alpha_l+\alpha_{-l}>0\}.$
\end{enumerate}

In the case (b) we will obtain that the $\bZ_l$-coordinate of the sum is equal to $4r^{M-1}_l$ for $l<M$ and $0$ otherwise.
In particular, the $\bZ_{M-1}$-coordinate is nonzero, which,  when compared with the equality \eqref{eq:sumcoord}, implies that $\cont$ intersects $\cT$ at the level $\beta_{M-1}$ or $-\beta_{M-1}$.
 Then we repeat the reasoning, studying the sum \eqref{eq:sumM} for $M-1$ instead of $M$.

Note that to compute the indices $\BIF(\beta_l), \BIF(-\beta_l)$ in \eqref{eq:sumM}, we use Lemma \ref{lem:coordinates}. The values of these indices depend on parities of dimensions of $\bV(l)$ and $\cH^N_l$. Since we have already observed that $\dim \bV(\mu)$ is odd, in consecutive steps we have to take into account only parities of  dimensions of spaces $\cH^N_l$. Therefore, in these steps we consider two cases of such parity, each time considering values of the sum \eqref{eq:sumM}.

Step 1.   Since, from the definition of $\mu$, we know that $\cont$ intersects $\cT$ at the level $\beta_{\mu}$ or $-\beta_{\mu}$, we start our considerations taking $M=\mu$  in \eqref{eq:sumM}, i.e. considering $\alpha_{\mu}\BIF(\beta_{\mu})+\alpha_{-\mu}\BIF(-\beta_{\mu})$ . Note that since $n_-=2n_+$ and the numbers $n_+$, $\dim\bV({\mu})$ are odd, by Lemma \ref{lem:coordinates} we get for $l\leq \mu$:
\[\BIF(\beta_{\mu})_{\bZ_{l}}= -n_- k^\mu_l=-2n_+ k^\mu_l,\]
 \[\BIF(-\beta_{\mu})_{\bZ_{l}}= n_+ \big((1-(-1)^{\dim \cH^N_\mu})r^{\mu-1}_l+k^\mu_l \big)\]
 and hence
\[
\alpha_{\mu}\BIF(\beta_{\mu})_{\bZ_{l}}+\alpha_{-{\mu}}\BIF(-\beta_{\mu})_{\bZ_{l}}= n_+\left(k^\mu_l(-2\alpha_{\mu}+\alpha_{-{\mu}})+r^{\mu-1}_l(1-(-1)^{\dim \cH^N_\mu})\alpha_{-{\mu}}\right),
\]

Since $r^{\mu-1}_{\mu}=0$, this sum equals $0$ for $l=\mu$ if and only if  $-2\alpha_{\mu}+\alpha_{-\mu}=0$, i.e. $\alpha_\mu=1, \alpha_{-\mu}=2.$

In the case when  $\dim \cH^N_\mu$ is even,  with these values of $\alpha_{\pm \mu}$ we obtain that
\[
\alpha_{\mu}\BIF(\beta_{\mu})_{\bZ_{l}}+\alpha_{-{\mu}}\BIF(-\beta_{\mu})_{\bZ_{l}}=0
\]
for every $l\leq\mu$.
Moreover, from Lemma \ref{lem:coordinates} it easily follows that
\[\alpha_{\mu}\BIF(\beta_{\mu})_{SO(2)}+\alpha_{-{\mu}}\BIF(-\beta_{\mu})_{SO(2)}=0\]
and therefore the sum \eqref{eq:sumM} equals $\Theta \in U(SO(2))$, so the possibilities (i), (ii) mentioned above  can occur.
Taking into consideration that $\alpha_\mu=1, \alpha_{-\mu}=2$, we obtain that if (i) occurs, then (A1) is proved for all $l$, whereas if (ii) occurs, the assertion is proved for $l\in\{\mu_1+1,\ldots, \mu\}$, where $\mu_1$ is defined in (ii).

Therefore to finish the proof of (A1) we have to consider only the case when $\dim \cH^N_{\mu}$ is odd. In this case we obtain
\[
\alpha_{\mu}\BIF(\beta_{\mu})_{\bZ_{l}}+\alpha_{-{\mu}}\BIF(-\beta_{\mu})_{\bZ_{l}}= 4r^{\mu-1}_l\neq0
\]
for every $l<\mu$.  Since, from Lemma \ref{lem:coordinates} $\BIF(\pm\beta_{\mu-1})_{\bZ_{\mu-1}}\neq 0$ and $\BIF(\pm\beta_{l})_{\bZ_{\mu-1}}=0$ for $l<\mu-1$, the equality \eqref{eq:sumcoord} (with $l=\mu-1$) implies that at least one of the numbers $\alpha_{\mu-1}, \alpha_{-(\mu-1)}$ is nonzero,  i.e. the continuum intersects $\cT$ at the level $\beta_{\mu-1}$ or $-\beta_{\mu-1}.$
 Therefore in the next step we continue considerations in this case, studying the sum \eqref{eq:sumM} for $M=\mu-1$.

Step 2.  Since $\dim \bV(\mu)$ and $\dim \cH^N_\mu$ are odd, it follows that $\dim \bV(\mu-1)$ is even. Therefore, by Lemma \ref{lem:coordinates}, we obtain:
\[\BIF(\beta_{\mu-1})_{\bZ_{l}}= -n_- k^{\mu-1}_l=-2n_+ k^{\mu-1}_l,\]
 \[\BIF(-\beta_{\mu-1})_{\bZ_{l}}
= n_+ \big(((-1)^{\dim \cH^N_{\mu-1}}-1)r^{\mu-2}_l-k^{\mu-1}_l \big)\]
and hence (using the equality $r^{\mu-1}_l=r^{\mu-2}_l+k^{\mu-1}_l$)
\begin{eqnarray*}
\sum_{j=\mu-1}^\mu \left(\alpha_j \BIF(\beta_{j})_{\bZ_{l}}+\alpha_{-j} \BIF(-\beta_j)_{\bZ_{l}}\right)=\\
=n_+ \left(k^{\mu-1}_l(4-2\alpha_{\mu-1} -\alpha_{-(\mu-1)})+r^{\mu-2}_l(4+((-1)^{\dim \cH^N_{\mu-1}}-1)\alpha_{-(\mu-1)} \right).
\end{eqnarray*}

Since $r^{\mu-2}_{\mu-1} =0$, the above sum equals 0 for $l=\mu-1$ if  and only if $4-2\alpha_{\mu-1} -\alpha_{-(\mu-1)}=0$, which is possible only if $\alpha_{\mu-1}=1, \alpha_{-(\mu-1)}=2$ or $\alpha_{\mu-1}=2, \alpha_{-(\mu-1)}=0$.

 With these values of $\alpha_{\mu-1}, \alpha_{-(\mu-1)}$ in the case of even $\dim \cH^N_{\mu-1}$ we can reason as for odd $\dim \cH^N_{\mu}$ and  obtain that $\cont$ intersects
$\cT$ at the level $\beta_{\mu-2}$ or $-\beta_{\mu-2}$. Moreover, this reasoning remains valid for subsequent levels up to $\tilde{\mu}\leq\mu-2$ such that $\dim \cH^N_{\tilde{\mu}}$ is odd and $\dim \cH^N_{l}$ is even for every $l\in\{\tilde{\mu}+1,\ldots,\mu-1\}$. Note that in this case we obtain the proof of (A1) for such values of $l$.

In particular, if $\tilde{\mu}=1$, we can analyse two cases of $\dim \cH^N_1$ (and, as a consequence, $\dim \bV(1)$). It can be easily obtained that in this case only (i) can occur and (A1) holds for all values of $l$.

Therefore without loss of generality we can assume that $\dim \cH^N_{\mu-1}$ is odd. Note that we have to analyse two possible situations of the value $\alpha_{-(\mu-1)}$. If  $\alpha_{-(\mu-1)}=2$ reasoning as in Step 1 (for even value of $\dim \cH^N_{\mu}$) we obtain possibilities (i) and (ii). Moreover, the assertion (A1) holds for all $l$ in the case (i) and for $l\in\{\mu_1+1,\ldots \mu\}$ in (ii).

If $\alpha_{-(\mu-1)}=0$ the reasoning is similar to that of odd value of $\dim \cH^N_{\mu}.$ More precisely, it is easy to prove that the continuum must intersect $\cT$ at the level $\beta_{\mu-2}$ or $-\beta_{\mu-2}.$ Then we continue considerations of the sum $\eqref{eq:sumM}$ with $M=\mu-2$.

 Step 3. We have
\begin{eqnarray*}
\sum_{j=\mu-2}^\mu \left(\alpha_j \BIF(\beta_{j})_{\bZ_{l}}+\alpha_{-j} \BIF(-\beta_j)_{\bZ_{l}}\right)=\\=
n_+ \left(k^{\mu-2}_l(4-2\alpha_{\mu-2}+\alpha_{-(\mu-2)}) +r^{\mu-3}_l(4+(1-(-1)^{\dim \cH^N_{\mu-2}})\alpha_{-(\mu-2)} \right).
\end{eqnarray*}
For $l=\mu-2$ the above expression is equal to $0$ if and only if $4-2\alpha_{\mu-2}+\alpha_{-(\mu-2)}=0$, which implies $\alpha_{\mu-2}=2, \alpha_{-(\mu-2)}=0$. Then for $l<\mu-2$ this expression is equal to $4r_l^{\mu-2}.$

According to the parity of $\dim \cH^N_{\mu-2}$ we can repeat one of the previous reasonings, in each case obtaining the proof of (A1).

Hence we have proved that (A1) holds in every case. Now we are in a position to prove the assertion of the theorem.
Suppose that (C1) is satisfied, i.e. $n_-=2n_+$ and $n_+$ is odd. Fix $\beta_m\in \sigma(-\Delta_{S^{N-1}})\setminus\{0\}$ and suppose that all four continua $\cC(\tilde{u}_i,\pm\beta_{m})$, $i=1,2$, are bounded. Then by (A1) either
\begin{enumerate}
\item $\cC(\tilde{u}_1,\beta_{m})=\cC(\tilde{u}_1,-\beta_{m})=\cC(\tilde{u}_2,-\beta_{m})$ and $\cC(\tilde{u}_1,\beta_{m})$ is disjoint with $\cC(\tilde{u}_2,\beta_{m})$
\item or $\cC(\tilde{u}_1,\beta_{m})=\cC(\tilde{u}_2,\beta_{m})$ and $\cC(\tilde{u}_1,\beta_{m})$ is disjoint with $\cC(\tilde{u}_i,-\beta_{m})$, $i=1,2$.
\end{enumerate}
Note that in the former case $\cC(\tilde{u}_2,\beta_{m})$ is unbounded. Indeed, if it was bounded, then by (A1) it would have to intersect $G(\tilde{u}_1)\times\{\beta_m\}$ or $G(\tilde{u}_i)\times\{-\beta_m\}$, $i=1,2$, which contradicts (1). Analogously, in the latter case we show that both the continua $\cC(\tilde{u}_i,-\beta_{m})$, $i=1,2$, are unbounded.

In the case (C2), repeating the above reasoning, we prove that the assertion follows from (A2).
\end{proof}

Now we are going to consider the continua bifurcating from $G(\tilde{u}_i)\times\{0\}$, $i=1,2$. Recall that from Lemma \ref{lem:existenceofacontinuum}
 it follows that if $n_-, n_+$ are of different parity, then the global bifurcation phenomenon occurs from $G(\tilde{u}_i)\times\{0\}$, $i=1,2$. In the following theorem we show that the bifurcating continua are unbounded.

\begin{Theorem}\label{thm:unbounded0}
 Consider the system \eqref {eq:system} with the potential $F$ and $u_1, u_2$ satisfying the assumptions (a1), (a2'), (a3)-(a5) and (a6'). If $n_-, n_+$ are of different parity, then the continua $\cC(\tilde{u}_i,0)$, $i=1,2$, are unbounded.
\end{Theorem}

\begin{proof}
Recall that $\BIF_{SO(2)}(\tilde{u}_1, 0)=\BIF_{SO(2)}(\tilde{u}_2, 0).$ As before, to simplify notations, we denote these values by $\BIF(0)$.

To show unboundedness of the bifurcating continua, we will consider two cases:
\begin{enumerate}
\item[(S1)] $n_- \neq 2n_+$ and $n_+ \neq 2n_-$
\item[(S2)] $n_- = 2n_+$ (for $n_+$ odd) or $n_+ = 2n_-$ (for $n_-$ odd).
\end{enumerate}
 Assume that (S1) is satisfied and suppose that the continuum $\cC(\tilde{u}_i,0)$ is bounded. If $\cC(\tilde{u}_i,0)\cap \cT \subset (G(\tilde{u}_1) \times \{0\}) \cup (G(\tilde{u}_2) \times \{0\})$, then by Theorem \ref{thm:AltRab} $$\alpha_0 \BIF(0) = \Theta \in U(SO(2)),$$
where $\alpha_0 \in \{1,2\}$ is the number of intersections of $\mathcal{C}(\tilde{u}_i, 0)$ with $\cT$ at the level $0$.
But since $\BIF(0) \neq \Theta$, this is impossible. This implies that $\mathcal{C}(\tilde{u}_i, 0)$ intersects $\cT$ at some level  $\lambda_{m_0}\in \Lambda\setminus\{0\}$. On the other hand, by Theorem \ref{thm:unbounded1}, the continua $\mathcal{C}(\tilde{u}_j, \lambda_{m_0}), j=1,2$, are unbounded, which proves the assertion in the case (S1).

In the case (S2) we will consider only the situation $n_-=2n_+$ and $n_+$ is odd, the analysis in the latter case is analogous. Suppose that $\mathcal{C}(\tilde{u}_i, 0)$ is bounded and, as in the previous case, observe that $\mathcal{C}(\tilde{u}_i, 0)$ intersects $\cT$ at some level $\lambda_{m_0}\in \Lambda\setminus\{0\}$.

From the proof of Theorem \ref{thm:unbounded2} it follows, that there exist $\mu >0$ and $M>0$ such that $\mathcal{C}(\tilde{u}_i, 0)$ intersects $\cT$ at the level $\beta_{\mu}$ or $-\beta_{\mu}$ and
$$\mathcal{C}(\tilde{u}_i, 0)\cap\cT\subset (G(\tilde{u}_1)\cup G(\tilde{u}_2))\times\{0,\pm\beta_M,\pm\beta_{M+1},\ldots,\pm\beta_{\mu}\}.$$
Recall that $\alpha_{\pm j}$ denotes the number of intersections of $\mathcal{C}(\tilde{u}_i, 0)$ with $\cT$ at a level $\pm \beta_j$. From the proof of Theorem \ref{thm:unbounded2} it follows that
\begin{equation*}
    \sum_{j=M}^{\mu} \left(\alpha_{j} \BIF(\beta_{j})+\alpha_{-j} \BIF(-\beta_{j})\right)=\Theta\in U(SO(2)).
\end{equation*}
Comparing the above equality with the equality from Theorem \ref{thm:AltRab} we obtain $\alpha_0 \BIF(0)=0$, a contradiction, which finishes the proof.
\end{proof}

\begin{Remark}
Since $(-1)^{\dim \cH^N_m}=(-1)^{k^m_0}$ and $k^m_0={{m+(N-3)}\choose{N-3}}$, the parity of $\dim \cH^N_m$ can be determined by applying the definition of the binomial coefficient. In particular the numbers $k^m_0$ form the $(N-2)$th downward diagonal of the Pascal triangle. Therefore, for example, for $N=3$ all $\cH^3_m$ are odd dimensional, whereas for $N=4$ all $\cH^4_{2j}$ are odd dimensional, while $\dim\cH^4_{2j-1}$ is even, $j=1,2,\ldots$.
\end{Remark}

\section{Appendix}
\subsection{Representations of the group $SO(2)$}\label{sec:representations}

Throughout our paper we consider some spaces as representations of the group $SO(2)$. Below we recall notations and formulate some facts from the theory of such representations.

Let $m \in \bN$. Denote by $\bR[1,m]$ the two-dimensional $SO(2)$-representation with the $SO(2)$-action given by $$\left(\left[\begin{array}{cc} \cos \varphi &-\sin \varphi\\ \sin \varphi &\cos \varphi\end{array}\right], \left[ \begin{array}{c} x\\y\end{array}\right]\right) \mapsto \left[ \begin{array}{cc}\cos m\varphi &-\sin m\varphi\\ \sin m\varphi &\cos m\varphi\end{array}\right]\left[\begin{array}{c} x\\y \end{array}\right]. $$
Moreover, by $\bR[k,m]$ we denote the direct sum of $k$ copies of the representation $\bR[1,m]$ and by $\bR[k,0]$ - the trivial $k$-dimensional $SO(2)$-representation.

It is known (see \cite{Adams}) that any finite-dimensional orthogonal $SO(2)$-representation is $SO(2)$-equivalent (we write $\approx_{SO(2)}$) to a representation of the form $\bR[k_0,0]\oplus \bR[k_1,m_1]\oplus \ldots \oplus \bR[k_r,m_r],$ for $k_0, \ldots, k_m \in \bN \cup \{0\}, m_1, \ldots, m_r \in \bN$.

In the following we describe such a decomposition for the eigenspaces of the Laplace-Beltrami operator, i.e. the spaces $\cH^N_m$, $m=0,1,\ldots$. Recall, that each $\cH^N_m$ can be represented as a linear space of harmonic, homogeneous polynomials of $N$ independent variables, of degree $m$, restricted to the sphere $S^{N-1}$. For variables written in spherical coordinates, $(\theta_1, \ldots, \theta_{N-1})$, where $0\leq \theta_1<2\pi$ and $0\leq \theta_k<\pi$ for $k \neq 1$, an orthonormal basis of $\cH^N_m$ is given by
\begin{equation}\label{eq:baza}
C_M(\theta_2, \ldots, \theta_{N-1}) \cos (m_{N-2} \theta_1), C_M(\theta_2, \ldots, \theta_{N-1}) \sin(m_{N-2} \theta_1),
\end{equation}
 for all $M=(m_0, \ldots, m_{N-3}, m_{N-2}), m=m_0 \geq m_1 \geq \ldots \geq m_{N-2} \geq0$, where $C_M$ are functions defined with the use of the Gegenbauer polynomials, see \cite{Vilenkin}, chapter IX, for details.

We define the action of $SO(2)$ on $\cH^N_m$ by
\begin{equation*}
\begin{split}
&(g(\varphi),C_M(\theta_2, \ldots, \theta_{N-1}) \cos (m_{N-2} \theta_1) ) \mapsto C_M(\theta_2, \ldots, \theta_{N-1}) \cos (m_{N-2} (\theta_1-\varphi)),\\
&(g(\varphi),C_M(\theta_2, \ldots, \theta_{N-1}) \sin (m_{N-2} \theta_1) ) \mapsto C_M(\theta_2, \ldots, \theta_{N-1}) \sin (m_{N-2} (\theta_1-\varphi)),
\end{split}
\end{equation*}
where $g(\varphi)=\left[\begin{array}{cc} \cos \varphi &-\sin \varphi\\ \sin \varphi &\cos \varphi\end{array}\right] \in SO(2).$ Therefore, for fixed $M$,

\begin{equation}\label{eq:dzialanie}
span_{\bR} \{C_M(\theta_2, \ldots, \theta_{N-1}) \cos (m_{N-2} \theta_1), C_M(\theta_2, \ldots, \theta_{N-1}) \sin (m_{N-2} \theta_1) \} \approx _{SO(2)} \bR[1,m_{N-2}].
\end{equation}

From the above formula, there exist numbers $k_0^m, \ldots, k_m^m \geq 0$ such that
\begin{equation}\label{eq:rozklad}
\cH^N_m \approx_{SO(2)} \bR[k_0^m,0] \oplus \bR[k_1^m, 1] \oplus \ldots \oplus \bR[k^m_m,m].
\end{equation}

To find the values of $k_0^m, \ldots, k_m^m$ we use the formula \eqref{eq:baza}. Note that from \eqref{eq:dzialanie} we obtain that to determine the number $k_j^m$ we have to count the number of sequences $M=(m_0, \ldots, m_{N-3}, m_{N-2})$ with $m_{N-2}=j$. Therefore, we are interested in the number of sequences  satisfying  $m=m_0 \geq m_1 \geq \ldots \geq m_{N-2} =j.$ To compute this number, we use the
one-to-one correspondence of such sequences with sequences $m=n_0\geq n_1>n_2> \ldots >n_{N-3}\geq n_{N-2}=j-(N-4)$ given by $(n_0, n_1, \ldots, n_{N-2})=(m_0, m_1, m_2-1, \ldots, m_{N-3}-(N-4), m_{N-2}-(N-4))$. Since the values $n_0$ and $n_{N-2}$ are fixed, each sequence $(n_0, \ldots, n_{N-2})$ is uniquely given by an $(N-3)$-element subset of the set $\{j-(N-4), \ldots, m\}$. Therefore the number of sequences equals ${m+(N-3)-j}\choose{N-3}$.

In particular, for $j=m$ we obtain $k^m_m={{N-3}\choose {N-3}} =1$ and for $j=m-1$ we obtain $k^m_{m-1}={{N-2}\choose{N-3}} = N-2.$

Recall that $\bV(m)=\bigoplus_{k=0}^m \cH^N_k$ and note that
\begin{equation}\label{eq:v(m)}
\bV(m)\approx_{SO(2)}\bR[r^m_0,0] \oplus \bR[r^m_1,1] \oplus \ldots \oplus \bR[r^m_{m-1},m-1] \oplus \bR[r^m_m,m],
\end{equation}
where $r^m_l=1+k^{l+1}_l+\ldots+k^{m}_l$ for $l\leq m$. Put $r^m_l=0$ for $l>m$. Moreover, it is easy to see that
$$(-1)^{k^m_0}=(-1)^{\dim \cH^N_m},\ \ \ (-1)^{r^m_0}=(-1)^{\dim \bV(m)}.$$

\subsection{Euler ring}\label{sec:Euler}
The bifurcation index which we use in our paper is an element of the so called Euler ring $U(G)$ for $G$ being a compact Lie group, see \cite{TomDieck1, TomDieck} for the definition. We denote this ring by $(U(G),+, \star)$ and $\chi_G(X)\in U(G)$ stands for the $G$-equivariant Euler characteristic of a pointed $G$-CW-complex $X$. The unit in $U(G)$ is $\bI=\chi_G(G/G^+)$, where for a $G$-CW-complex without a base point we denote by $X^+$ a pointed $G$-CW-complex $X \cup \{\ast\}$. The zero element in $U(G)$ is denoted by $\Theta$.

Denote by $\sub[G]$ the set of conjugacy classes $(H)$ of closed subgroups of the group $G$. The generators of $U(G)$ may be indexed by elements of $\sub[G]$. More precisely, there holds the following fact:

\begin{Fact}
$(U(G),+)$ is a free abelian group with the basis $\chi_G(G/H^+)$ for $(H)\in \sub[G]$. Therefore one can identify $U(G)$ with the $\bZ$-module $\oplus_{(H)\in \sub[G]}\bZ$, see Corollary IV.1.9 of \cite{TomDieck}.
\end{Fact}

In our considerations an important role is played by the group $SO(2)$. Note that since this is an abelian group, each conjugacy class of its closed subgroup consists of one element and therefore we can identify $\sub[SO(2)]$  with $\{SO(2), \bZ_1,\bZ_2, \ldots\}$. Taking this into consideration, we identify $U(SO(2))$ with $\bZ\oplus \bigoplus_{k=1}^{\infty}\bZ$ and we denote the elements of $U(SO(2))$ by $(\alpha_{SO(2)}, \alpha_{\bZ_1}, \alpha_{\bZ_2}, \ldots).$

The addition and multiplication in $U(SO(2))$ are given by
\begin{equation*}
\begin{split}
(\alpha_{SO(2)}, \alpha_{\bZ_1}, \alpha_{\bZ_2}, \ldots) &+ (\beta_{SO(2)}, \beta_{\bZ_1}, \beta_{\bZ_2}, \ldots)= (\alpha_{SO(2)}+\beta_{SO(2)}, \alpha_{\bZ_1}+\beta_{\bZ_1},\alpha_{\bZ_2}+\beta_{\bZ_2},\ldots)\\
(\alpha_{SO(2)}, \alpha_{\bZ_1}, \alpha_{\bZ_2}, \ldots) &\star (\beta_{SO(2)}, \beta_{\bZ_1}, \beta_{\bZ_2}, \ldots)=\\&=(\alpha_{SO(2)}\beta_{SO(2)}, \alpha_{SO(2)}\beta_{\bZ_1}+\alpha_{\bZ_1}\beta_{SO(2)},\alpha_{SO(2)}\beta_{\bZ_2}+\alpha_{\bZ_2}\beta_{SO(2)},\ldots).
\end{split}
\end{equation*}
Moreover, from the multiplication formula it easily follows that an element of the form $(\alpha_{SO(2)}, \alpha_{\bZ_1}, \alpha_{\bZ_2}, \ldots)$ is invertible in $U(SO(2))$ if and only if $\alpha_{SO(2)}=\pm1.$ In this case $$(\alpha_{SO(2)}, \alpha_{\bZ_1}, \alpha_{\bZ_2}, \ldots)^{-1}=(\alpha_{SO(2)}, -\alpha_{\bZ_1}, -\alpha_{\bZ_2}, \ldots).$$

In the computations of the bifurcation indices, we use the following formula, which is a corollary from the above considerations:
\begin{equation}\label{eq:powersinU(SO(2))}
\left.\begin{array}{ll}
(\alpha_{SO(2)},\alpha_{\bZ_1},\alpha_{\bZ_2},\ldots)^N\star ((\beta_{SO(2)},\beta_{\bZ_1},\beta_{\bZ_2},\ldots)^N-(1,0,\ldots))=\\
=(\alpha_{SO(2)}^N(\beta_{SO(2)}^N-1),Na_{SO(2)}^{N-1}\alpha_{\bZ_1}(\beta_{SO(2)}^N-1)+\alpha_{SO(2)}^NNb_{SO(2)}^{N-1}\beta_{\bZ_1}),\ldots,\\
Na_{SO(2)}^{N-1}\alpha_{\bZ_l}(\beta_{SO(2)}^N-1)+\alpha_{SO(2)}^NNb_{SO(2)}^{N-1}\beta_{\bZ_l},\ldots)=\\
=(\alpha_{SO(2)}^N(\beta_{SO(2)}^N-1),Na_{SO(2)}^{N-1}(\alpha_{\bZ_1}\beta_{SO(2)}^N-\alpha_{\bZ_1}+\alpha_{SO(2)}\beta_{SO(2)}^{N-1}\beta_{\bZ_1}), \ldots,\\
Na_{SO(2)}^{N-1}(\alpha_{\bZ_l}\beta_{SO(2)}^N-\alpha_{\bZ_l}+\alpha_{SO(2)}\beta_{SO(2)}^{N-1}\beta_{\bZ_l}),\ldots).
\end{array}\right.
\end{equation}

\subsection{Equivariant degree}\label{sec:degree}
Let $\bV$ be a finite dimensional, orthogonal representation of a compact Lie group $G$, and let $\varphi \in C^1(\bV,\bR)$ be a $G$-invariant function. For such $\bV, \varphi$ and an open, bounded, $G$-invariant set $\Omega \subset \bV$ such that $\partial \Omega \cap (\nabla \varphi)^{-1}(0) = \emptyset$, there has been defined the degree $\degg(\nabla \varphi, \Omega)$, being an element of the Euler ring $U(G)$, see \cite{Geba}. This degree has properties analogous to these of the Brouwer degree, i.e. the excision, additivity, linearisation and homotopy invariance properties, see \cite{Geba}.  Moreover, there holds the product formula for the degree, namely the degree of the product map is the product (in Euler ring) of the degrees, see \cite{GolRyb2013}.

In the case $G=SO(2)$, following the notation of the coordinates of an element of $U(SO(2))$, we write:
\begin{equation*}
\begin{split}
&\degso(\nabla \varphi, \Omega)=\\&(\degso_{SO(2)}(\nabla \varphi, \Omega),\degso_{\bZ_1}(\nabla \varphi,\Omega), \degso_{\bZ_2}(\nabla \varphi,\Omega), \ldots ).
\end{split}
\end{equation*}

Moreover, in this case, we can obtain an explicit formula for the degree in some situations. In particular, if $\bV$ is a finite dimensional, orthogonal $SO(2)$-representation equivalent to $\bR[k_0,0]\oplus\bR[k_1,m_1]\oplus\ldots\oplus\bR[k_r,m_r]$, then
\begin{equation}\label{eq:stopienId}
\begin{array}{ll}
\nabla_{SO(2)}\text{-}\deg_{H}(-Id,B(\bV))=
\left\{
\begin{array}{ll}
(-1)^{k_0}&\text{ for } H=SO(2)\\
(-1)^{k_0+1}k_i&\text{ for } H=\bZ_i, i \in \{m_1, \ldots, m_r\}\\
0&\text{ for } H\notin\{SO(2),\bZ_{m_1},\ldots,\bZ_{m_r}\}.\\
\end{array}
\right.
\end{array}
\end{equation}
In the infinite dimensional situation, there is defined a generalisation of $\degg(\cdot, \cdot)$, denoted by the same symbol $\degg(\cdot, \cdot)$, for the case of invariant strongly indefinite functionals. Let $\bH$ be an infinite dimensional Hilbert space, which is an orthogonal $G$-representation. Moreover, we assume that $\bH$ admits an approximation scheme $\{\pi_n\colon \bH \to \bH: n \in \bN \cup \{0\}\}$, see \cite{GolRyb2011}, and we put $\bH^n=\pi_n(\bH)$. We consider an open, bounded, $G$-invariant set $\Omega \subset \bH$ and a $G$-invariant functional $\Phi \in C^1(\bH, \bR)$ of the form $\Phi(u)=\frac{1}{2} \langle Lu,u \rangle - \eta(u)$, where $L\colon \bH \to \bH$ is a linear, bounded, self-adjoint, $G$-equivariant Fredholm operator of index $0$, such that $\ker L=\bH^0$  and $\pi_n \circ L =L \circ \pi_n$ for all $n \in \bN \cup \{0\}$ and $\nabla \eta \colon \bH \to \bH$ is a $G$-equivariant completely continuous operator. For such $\Omega$ and $\Phi$, if moreover $(\nabla \Phi)^{-1}(0) \cap \partial \Omega = \emptyset,$ we can define the degree $\degg(\nabla \Phi, \Omega)\in U(G)$ by the formula:
\begin{equation*}
\degg(\nabla \Phi, \Omega)=\degg(L, B(\bH^n \ominus \bH^0))^{-1} \star \degg (L-\pi_n \circ \nabla \eta, \Omega \cap \bH^n),
\end{equation*}
see \cite{GolRyb2011} for details. This degree also has the  properties of excision, additivity, linearisation and homotopy invariance.
For the general theory of the equivariant degree we refer the reader to \cite{BalKra}, \cite{BalKraRyb}, \cite{Geba}, \cite{GolRyb2011}, \cite{Ryb2005milano}.

\end{document}